\documentclass{article}
\usepackage{amsmath,amsthm,amsfonts,graphicx}
\usepackage{multirow}
\usepackage{slashbox}
\usepackage{multirow}
\def\smallddots{\mathinner{\raise7pt\hbox{.}\raise4pt\hbox{.}\raise1pt\hbox{.}}}
\def\smallsdots{\mathinner{\raise1pt\hbox{.}\raise4pt\hbox{.}\raise7pt\hbox{.}}}

\DeclareMathOperator{\diag}{diag}

\numberwithin{equation}{section}
\numberwithin{table}{section}
\newtheorem{theorem}{Theorem}[section]
\newtheorem{lemma}{Lemma}[section]

\newtheorem{fact}{Fact}[section]

\newtheorem{definition}{Definition}[section]

\newtheorem{remark}{Remark}[section]

\begin{document}

\title{\bf Condition Numbers of
Random  Toeplitz and Circulant Matrices
\thanks {The results of this paper have been presented at the 
 ACM-SIGSAM International 
Symposium on Symbolic and Algebraic Computation (ISSAC '2011), San Jose, CA, 2011,
the
3nd International Conference on Matrix Methods in Mathematics and 
Applications (MMMA 2011) in
Moscow, Russia, June 22-25, 2011, 
the 7th International Congress on Industrial and Applied Mathematics 
(ICIAM 2011), in Vancouver, British Columbia, Canada, July 18-22, 2011,
the SIAM International Conference on Linear Algebra,
in Valencia, Spain, June 18-22, 2012, and 
the Conference on Structured Linear and Multilinear Algebra Problems (SLA2012),
in  Leuven, Belgium, September 10-14, 2012}}

\author{Victor Y. Pan$^{[1, 2],[a]}$ and Guoliang Qian$^{[2],[b]}$
\and\\
$^{[1]}$ Department of Mathematics and Computer Science \\
Lehman College of the City University of New York \\
Bronx, NY 10468 USA \\
$^{[2]}$ Ph.D. Programs in Mathematics  and Computer Science \\
The Graduate Center of the City University of New York \\
New York, NY 10036 USA \\
$^{[a]}$ victor.pan@lehman.cuny.edu \\
http://comet.lehman.cuny.edu/vpan/  \\
$^{[b]}$ gqian@gc.cuny.edu \\
} 
 \date{}

\maketitle


\begin{abstract}
Estimating the condition numbers of random structured matrices
is a well known challenge (cf.  \cite{SST06}),
linked to the design  of efficient randomized matrix algorithms 
in \cite{PGMQ},  \cite{PIMR10},  \cite{PQ10},  \cite{PQ12},  
\cite{PQZa},  \cite{PQa},
  \cite{PQZb},   \cite{PQZC},  \cite{PY09}.
We deduce such estimates for Gaussian random Toeplitz and circulant
matrices.
The former estimates can be surprising because
the condition numbers
 grow exponentially in $n$ as $n\rightarrow \infty$
for some large and important classes
of  $n\times n$ 
Toeplitz matrices \cite{BG05}, 
whereas we prove the opposit for 
Gaussian random Toeplitz matrices.
Our formal estimates are in good accordance 
with our numerical tests,
except that circulant matrices tend to be even
better conditioned according to the tests than according to our formal
study.  
\end{abstract}





\section{Introduction}\label{sintro}


It is well known that random matrices tend to be well conditioned
 \cite{D88},  \cite{E88},  \cite{ES05},
\cite{CD05}, \cite{SST06}, \cite{B11}, 
and this property can be exploited for advancing
matrix computations (see e.g., 
\cite{PGMQ},  \cite{PIMR10},  \cite{PQ10},  \cite{PQ12},  
\cite{PQa},  \cite{PQZa},
  \cite{PQZb},   \cite{PQZC}
\cite{PY09}).
Exploiting matrix structure in these
applications was supported empirically in the latter papers 
and formally in \cite{T11}. An important step 
in this direction is the estimation of the condition
numbers of  
structured 
matrices 
 stated as a challenge in  \cite{SST06}.
We reply to this challenge 
by  estimating the condition numbers 
of Gaussian random Toeplitz and circulant
matrices both formally 
 (see Sections \ref{scgrtm} and \ref{scgrcm})
and experimentally
(see Tables \ref{nonsymtoeplitz}--\ref{tabcondcirc}).
Our study shows that 
Gaussian random Toeplitz  circulant
matrices do not tend to be ill conditioned
and the condition numbers of
Gaussian random circulant  $n\times n$
matrices tend to grow extremely slow as $n$ grows large.
Our numerical tests
(the contribution of the second author)
 are in good accordance 
with our formal estimates,
except that circulant matrices tended to be even
better conditioned in the tests than according to our formal
study.  
Our results on Toeplitz matrices are
quite surprising because
the condition numbers
 grow exponentially in $n$ as $n\rightarrow \infty$
for some large and important classes
of  $n\times n$ 
Toeplitz matrices \cite{BG05}, 
which is opposit  to the behavior of 
Gaussian random Toeplitz $n\times n$ matrices
as we proved and consistently observed in our tests.
Clearly, our study of  
Toeplitz matrices can be equally applied to Hankel 
 matrices. 


We organize our paper as follows.
We recall some definitions and basic results on general matrix computations
in the next 
section and on Toeplitz, Hankel and circulant matrices in Section \ref{stplc}.
We define Gaussian random matrices and study their
ranks and extremal singular values in Section \ref{srvrm}.
In Sections \ref{scgrtm} and \ref{scgrcm} we extend this study 
to Gaussian random 
 Toeplitz and  circulant  matrices, respectively.
 In Section \ref{sexp} we cover numerical tests,
which constitute the contribution of the second author.
In Section \ref{simprel} we recall some applications
of random circulant and  Toeplitz matrices, which
provide some implicit empirical
 support for our estimates for their condition 
numbers.
We end with conclusions in Section \ref{srel}.



\section{Some definitions and basic results}\label{sdef}


Except for Theorem \ref{thcpw}   
and its application in the proof of Theorem \ref{thcircsing}
we work in the field $\mathbb R$ of real numbers.
Next we recall some customary definitions of matrix computations
\cite{GL96}, \cite{S98}.
$A^T$ is the transpose  of a 
matrix $A$.
$||A||_h$ is its $h$-norm for 
$h=1,2,\infty$. We write $||A||$ to denote the 2-norm $||A||_2$.
We have
\begin{equation}\label{eqnorm12}
\frac{1}{\sqrt m}||A||_1\le||A||\le \sqrt n ||A||_1,~~||A||_1=||A^T||_{\infty},~~
||A||^2\le||A||_1||A||_{\infty}, 
\end{equation}
for an $m\times n$ matrix $A$,
\begin{equation}\label{eqnorm12inf}
||AB||_h\le ||A||_h||B||_h~{\rm for}~h=1,2,\infty~{\rm and~any~matrix~product}~AB.
\end{equation}



Define an {\em SVD} or {\em full SVD} of an $m\times n$ matrix $A$ of a rank 
 $\rho$ as follows,
\begin{equation}\label{eqsvd}
A=S_A\Sigma_AT_A^T.
\end{equation}
Here
$S_AS_A^T=S_A^TS_A=I_m$, $T_AT_A^T=T_A^TT_A=I_n$,
$\Sigma_A=\diag(\widehat \Sigma_A,O_{m-\rho,n-\rho})$, 
$\widehat \Sigma_A=\diag(\sigma_j(A))_{j=1}^{\rho}$,
$\sigma_j=\sigma_j(A)=\sigma_j(A^T)$
is the $j$th largest singular value of a matrix $A$
 for $j=1,\dots,\rho$, and we write
$\sigma_j=0$ for $j>\rho$. These values have 
the minimax property  
\begin{equation}\label{eqminmax}
\sigma_j=\max_{{\rm dim} (\mathbb S)=j}~~\min_{{\bf x}\in \mathbb S,~||{\bf x}||=1}~~~||A{\bf x}||,~j=1,\dots,\rho,
\end{equation}
where $\mathbb S$ denotes linear spaces  \cite[Theorem 8.6.1]{GL96}.

 
\begin{fact}\label{faccondsub} 
If $A_0$ is a 
submatrix of a 
matrix $A$, 
then
$\sigma_{j} (A)\ge \sigma_{j} (A_0)$ for all $j$.
\end{fact} 

 
\begin{proof}
\cite[Corollary 8.6.3]{GL96} implies 
the claimed bound
where $A_0$ is any block of columns of 
the matrix $A$. Transposition of a matrix and permutations 
of its rows and columns do not change singular values,
and thus we can extend the bounds to
all submatrices $A_0$.
\end{proof}


$A^+=T_A\diag(\widehat \Sigma_A^{-1},O_{n-\rho,m-\rho})S_A^T$ is the Moore--Penrose 
pseudo-inverse of the matrix $A$ of (\ref{eqsvd}), and
\begin{equation}\label{eqnrm+}
||A^+||=1/\sigma{_\rho}(A)
\end{equation}
 for 
a matrix $A$ of a rank $\rho$.  


$\kappa (A)=\frac{\sigma_1(A)}{\sigma_{\rho}(A)}=||A||~||A^+||$ is the condition 
number of an $m\times n$ matrix $A$ of a rank $\rho$. Such matrix is {\em ill conditioned} 
if $\sigma_1(A)\gg\sigma_{\rho}(A)$ and is {\em well conditioned}
otherwise. See \cite{D83}, \cite[Sections 2.3.2, 2.3.3, 3.5.4, 12.5]{GL96}, 
\cite[Chapter 15]{H02}, \cite{KL94}, and \cite[Section 5.3]{S98} 
on the estimation of the norms and condition numbers 
of nonsingular matrices. 


\section{Toeplitz, Hankel and $f$-circulant 
matrices}\label{stplc}


A {\em Toep\-litz} $m\times n$ matrix $T_{m,n}=(t_{i-j})_{i,j=1}^{m,n}$ 
(resp. Hankel matrices $H=(h_{i+j})_{i,j=1}^{m,n}$) 
is defined by its first row and first (resp. last) column, that is by
the vector $(t_h)_{h=1-n}^{m-1}$ (resp. $(h_g)_{g=2}^{m+n}$)
of dimension $m+n-1$. We write $T_n=T_{n,n}=(t_{i-j})_{i,j=1}^{n,n}$ 
(see equation (\ref{eqtz}) below).

${\bf e}_i$ is the $i$th coordinate vector of dimension $n$ for
$i=1,\dots,n$. The
  reflection matrix 
$J=J_n({\bf e}_n~|~\dots~|~{\bf e}_1)$ is the Hankel $n\times n$ matrix
defined by its first column ${\bf e}_n$ and its last column 
${\bf e}_1$. We have $J=J^T=J^{-1}$.

A lower {\em triangular Toep\-litz}  $n\times n$ matrix $Z({\bf t})=(t_{i-j})_{i,j=1}^n$
(where $t_k=0$ for $k<0$)
is defined by its first column ${\bf t}=(t_h)_{h=0}^{n-1}$. 
We write $Z({\bf t})^T=(Z({\bf t}))^T$.
$Z=Z_0=Z({\bf e}_2)$
is the downshift $n\times n$ matrix
 (see (\ref{eqtz})). We have 
$Z{\bf v}=(v_i)_{i=0}^{n-1}$ and
$Z({\bf v})=Z_0({\bf v})=\sum_{i=1}^{n}v_{i}Z^{i-1}$
for  ${\bf v}=(v_i)_{i=1}^n$ and $v_0=0$,
\begin{equation}\label{eqtz}
T_n=\begin{pmatrix}t_0&t_{-1}&\cdots&t_{1-n}\\ t_1&t_0&\smallddots&\vdots\\ \vdots&\smallddots&\smallddots&t_{-1}\\ t_{n-1}&\cdots&t_1&t_0\end{pmatrix},~Z=\begin{pmatrix}
        0   &       &   \dots    &   & 0\\
        1   & \ddots    &       &   & \\
        \vdots     & \ddots    & \ddots    &   & \vdots    \\
            &       & \ddots    & 0 &  \\
        0    &       &  \dots      & 1 & 0 
    \end{pmatrix}, ~Z_f=\begin{pmatrix}
        0   &       &   \dots    &   & f\\
        1   & \ddots    &       &   & \\
        \vdots     & \ddots    & \ddots    &   & \vdots    \\
            &       & \ddots    & 0 &  \\
        0    &       &  \dots      & 1 & 0 
    \end{pmatrix}.
\end{equation}
Combine the equations $||Z({\bf v})||_1=||Z({\bf v})||_{\infty}=||{\bf v}||_1$
 with  (\ref{eqnorm12}) to obtain 
\begin{equation}\label{eqttn}
||Z({\bf v})||\le ||{\bf v}||_1.
\end{equation}



\begin{theorem}\label{thgs}
Write $T_{k}=(t_{i-j})_{i,j=0}^{k-1}$ for  $k=n,n+1$. 

(a) Let the matrix $T_n$ be nonsingular and write 
${\bf p}=T_n^{-1}{\bf e}_1$ and ${\bf q}=T_n^{-1}{\bf e}_{n}$.
If
$p_{1}={\bf e}_1^T{\bf p}\neq 0$,
then
$p_{1}T_n^{-1}=Z({\bf p})Z(J{\bf q})^T-Z(Z{\bf q})Z(ZJ{\bf p})^T.$

In parts (b) and (c) below let the matrix $T_{n+1}$ be nonsingular and write 
$\widehat {\bf v}=(v_i)_{i=0}^n=T_{n+1}^{-1}{\bf e}_1$,
${\bf v}=(v_i)_{i=0}^{n-1}$, ${\bf v}'=(v_i)_{i=1}^{n}$,
$\widehat {\bf w}=(w_i)_{i=0}^n=T_{n+1}^{-1}{\bf e}_{n+1}$, 
${\bf w}=(w_i)_{i=0}^{n-1}$, and ${\bf w}'=(w_i)_{i=1}^{n}$.

(b) If $v_0\neq 0$, then the matrix $T_n$ is nonsingular and
$v_0T_n^{-1}=Z({\bf v})Z(J{\bf w'})^T-Z({\bf w})Z(J{\bf v}')^T$.

(c) If $v_n\neq 0$, then the matrix $T_{1,0}=(t_{i-j})_{i=1,j=0}^{n,n-1}$ is nonsingular and
$v_nT_{1,0}^{-1}=Z({\bf w})Z(J{\bf v'})^T-Z({\bf v})Z(J{\bf w}')^T$.
\end{theorem}
\begin{proof}
 See \cite{GS72} on parts (a) and (b);  see \cite{GK72} on part (c).
\end{proof}

$Z_f=Z+f{\bf e}_1^T{\bf e}_n$ for a scalar $f\neq 0$
denotes the
$n\times n$ matrix of
 $f$-{\em circular shift} (see (\ref{eqtz})).
An $f$-{\em circulant matrix} $Z_f({\bf v})=\sum_{i=1}^{n}v_iZ_f^{i-1}$ 
is a special Toep\-litz $n\times n$ matrix defined by its first column vector 
${\bf v}=(v_i)_{i=1}^{n}$ and a scalar $f$. 
$f$-circulant matrix is called {\em circulant} if $f=1$ and {\em skew circulant} if $f=-1$.
By replacing $f$ with $0$ we arrive at a lower triangular 
Toep\-litz matrix $Z({\bf v})$.
The following theorem implies that the inverses 
 (wherever they are defined) and pairwise  products of  
$f$-circulant  $n\times n$ matrices are $f$-circulant  and can be computed 
in $O(n\log n)$ flops. 

\begin{theorem}\label{thcpw} (See \cite{CPW74}.)
We have 
$Z_1({\bf v})=\Omega^{-1}D(\Omega{\bf v})\Omega.$
More generally, for any $f\ne 0$, we have
$Z_{f^n}({\bf v})=U_f^{-1}D(U_f{\bf v})U_f$
where
$U_f=\Omega D({\bf f}),~~{\bf f}=(f^i)_{i=0}^{n-1}$,
$D({\bf u})=\diag(u_i)_{i=0}^{n-1}$ for a vector ${\bf u}=(u_i)_{i=0}^{n-1}$,  
$\Omega=(\omega_n^{ij})_{i,j=0}^{n-1}$ is the $n\times n$ matrix of the 
discrete Fourier transform at $n$ points, 
$\omega_n={\rm exp}(\frac{2\pi}{n}\sqrt{-1})$ being 
a primitive $n$-th root of $1$, and $\Omega^{-1}=\frac{1}{n}(\omega_n^{-ij})_{i,j=0}^{n-1}=\frac{1}{n}\Omega^H$.
\end{theorem}

 {\em Hankel} $m\times n$ matrices $H=(h_{i+j})_{i,j=1}^{m,n}$ can be 
 defined equivalently
as the products $H=TJ_n$ or $H=J_mT$ of
$m\times n$ Toep\-litz 
matrices $T$ and the Hankel reflection matrices $J=J_m$ or $J_n$.
Note that $J=J^{-1}=J^T$ and obtain the following simple fact.
\begin{fact}\label{fath}
For $m=n$ we have 
 $T=HJ$, $H^{-1}=JT^{-1}$ and $T^{-1}=JH^{-1}$ if $H=TJ$,
whereas $T=JH$, $H^{-1}=JT^{-1}$ and $T^{-1}=H^{-1}J$ if $H=JT$.
Furthermore in both cases $\kappa (H)=\kappa (T)$.
\end{fact}
By using the equations above we can 
readily extend any Toep\-litz  matrix inversion algorithm 
 to Hankel 
 matrix inversion
and vice versa, 
preserving the flop count and condition numbers.
E.g. $(JT)^{-1}=T^{-1}J$, $(TJ)^{-1}=JT^{-1}$,
$(JH)^{-1}=H^{-1}J$ and $(HJ)^{-1}=JH^{-1}$.


\section{Gaussian random matrices and their
ranks}\label{srvrm}


\begin{definition}\label{defcdf}
$F_{\gamma}(y)=$ Probability$\{\gamma\le y\}$ (for a real random variable $\gamma$)
is the {\em cumulative 
distribution function (cdf)} of $\gamma$ evaluated at $y$. 
$F_{g(\mu,\sigma)}(y)=\frac{1}{\sigma\sqrt {2\pi}}\int_{-\infty}^y \exp (-\frac{(x-\mu)^2}{2\sigma^2}) dx$ 
for a Gaussian random variable $g(\mu,\sigma)$ with a mean $\mu$ and a positive variance $\sigma^2$,
and so   
\begin{equation}\label{eqnormal}
\mu-4\sigma\le y \le \mu+4\sigma~{\rm with ~a ~probability ~near ~1}.
\end{equation}
\end{definition}


\begin{definition}\label{defrndm}
A matrix (or a vector) is a {\em Gaussian random matrix (or vector)} with a mean 
$\mu$ and a positive variance $\sigma^2$ if it is filled with 
independent identically distributed Gaussian random 
variables, all having the mean $\mu$ and variance $\sigma^2$. 
$\mathcal G_{\mu,\sigma}^{m\times n}$ is the set of such
Gaussian  random  $m\times n$ matrices
(which are {\em standard} for $\mu=0$
and $\sigma^2=1$). By restricting this set 
to Toeplitz or $f$-circulant matrices we obtain the sets
$\mathcal T_{\mu,\sigma}^{m\times n}$ and
$\mathcal Z_{f,\mu,\sigma}^{n\times n}$ of
{\em Gaussian random Toep\-litz} 
and {\em Gaussian random $f$-circulant matrices}, 
respectively.  
\end{definition}


\begin{definition}\label{defchi}
$\chi_{\mu,\sigma,n}(y)$ is the cdf of the norm 
$||{\bf v}||=(\sum_{i=1}^n v_i^2)^{1/2}$ of a Gaussian random vector
${\bf v}=(v_i)_{i=1}^n\in \mathcal G_{\mu,\sigma}^{n\times 1}$. For 
 $y\ge 0$ we have 
$\chi_{0,1,n}(y)= \frac {2}{2^{n/2}\Gamma(n/2)}\int_{0}^yx^{n-1}\exp(-x^2/2) dx$ 
where $\Gamma(h)=\int_0^{\infty}x^{h-1}\exp(-x) dx$, $\Gamma (n+1)=n!$ for nonnegative integers $n$.
\end{definition}



The total degree of a multivariate monomial is the sum of its degrees
in all its variables. The total degree of a polynomial is the maximal total degree of 
its monomials.


\begin{lemma}\label{ledl} \cite{DL78}, \cite{S80}, \cite{Z79}.
For a set $\Delta$ of a cardinality $|\Delta|$ in any fixed ring  
let a polynomial in $m$ variables have a total degree $d$ and let it not vanish 
identically on this set. Then the polynomial vanishes in at most 
$d|\Delta|^{m-1}$ points. 
\end{lemma}


We assume that Gaussian random variables range 
over infinite sets $\Delta$,
usually over the real line or its interval. Then
the lemma implies that a nonzero polynomial vanishes with probability 0.
Consequently a square Gaussian random general, Toeplitz or circulant
matrix is nonsingular 
with probability 1
because its determinant is a polynomials
in the entries. 
Likewise rectangular
 Gaussian random general, Toeplitz and circulant 
matrices have full rank with probability 1.
Hereafter,
wherever this causes no confusion,  
we assume by default that
{\em Gaussian random 
 general, Toeplitz and circulant 
matrices  have full rank}.


\section{Extremal singular values of Gaussian random matrices}\label{ssvrm}

Besides  having full rank with probability 1,
Gaussian random matrices in Definition \ref{defrndm} are  likely to be well conditioned  
\cite{D88}, \cite{E88}, \cite{ES05}, \cite{CD05}, \cite{B11}, and 
even the sum $M+A$ for  $M\in \mathbb R^{m\times n}$ and 
 $A\in \mathcal G_{\mu,\sigma}^{m\times n}$ is  likely to
be well conditioned unless the ratio  $\sigma/||M||$ is small
or large 
\cite{SST06}. 

The following theorem 
states an upper bound 
proportional to $y$ on
 the cdf $F_{1/||A^+||}(y)$, that is  
on the probability that  the 
smallest positive singular value $1/||A^+||=\sigma_l(A)$ of a  Gaussian random matrix $A$ 
is less than a nonnegative scalar $y$ (cf. (\ref{eqnrm+}))
and consequently on the probability that the norm $||A^+||$
exceeds a positive scalar $x$.
The stated bound still holds if we replace the matrix $A$ by 
$A-B$ for any fixed matrix $B$, and
for $B=O_{m,n}$
the  bounds
can  be strengthened  
by a factor $y^{|m-n|}$ \cite{ES05}, \cite{CD05}.


\begin{theorem}\label{thsiguna} 
Suppose 
$A\in \mathcal G_{\mu,\sigma}^{m\times n}$, 
 $B\in \mathbb R^{m\times n}$,
$l=\min\{m,n\}$,  $x>0$, and $y\ge 0$. 
Then 
$F_{\sigma_l(A-B)}(y)\le 2.35~\sqrt l y/\sigma$, 
that is
$Probability \{||(A-B)^+||\ge 2.35x\sqrt {l}/\sigma\}\le 1/x$.
\end{theorem}
\begin{proof}
For $m=n$ this is \cite[Theorem 3.3]{SST06}. Apply
 Fact \ref{faccondsub} to extend it to any pair $\{m,n\}$.
\end{proof}


The following two theorems supply lower bounds
$F_{||A||}(z)$ and
$F_{\kappa (A)}(y)$
 on the probabilities 
that $||A||\le z$ 
and $\kappa(A)\le y$ for two scalars $y$ and $z$, 
respectively,
and a Gaussian random matrix $A$. 
We do not use the second theorem, but state it for the sake of completeness
and only for square $n\times n$ matrices $A$.
The theorems  
imply that  
the functions 
$1-F_{||A||}(z)$
and
$1-F_{\kappa (A)}(y)$ 
decay as 
$z\rightarrow \infty$ and
$y\rightarrow \infty$, respectively,
and that the two decays are exponential in $-z^2$ and  proportional 
to $\sqrt{\log y}/y$, respectively.
 For small values $y\sigma$ and a fixed $n$ 
the lower bound of Theorem \ref{thmsiguna}
becomes negative, in which case 
the theorem becomes trivial. 
Unlike Theorem \ref{thsiguna}, in both theorems we assume that $\mu=0$. 


\begin{theorem}\label{thsignorm} \cite[Theorem II.7]{DS01}.
Suppose $A\in \mathcal G_{0,\sigma}^{m\times n}$,
$h=\max\{m,n\}$  and
$z\ge 2\sigma\sqrt h$. 
Then $F_{||A||}(z)\ge 1- \exp(-(z-2\sigma\sqrt h)^2/(2\sigma^2))$, and so
the norm $||A||$ is  likely to have order $\sigma\sqrt h$. 
\end{theorem}


\begin{theorem}\label{thmsiguna}  \cite[Theorem 3.1]{SST06}.
Suppose  
$0<\sigma\le 1$,  
$y\ge 1$,  
 $A\in \mathcal G_{0,\sigma}^{n\times n}$. Then the matrix $A$
 has full rank with 
probability $1$ and 
$F_{\kappa (A)}(y)\ge 1-(14.1+4.7\sqrt{(2\ln y)/n})n/ (y\sigma)$.
\end{theorem}

  


 
\begin{proof}
See \cite[the proof of Lemma 3.2]{SST06}.
\end{proof}


\section{Extremal singular values of Gaussian random Toeplitz matrices}\label{scgrtm}


A  matrix 
$T_n=(t_{i-j})_{i,j=1}^n$
is the sum of two triangular 
Toeplitz matrices


\begin{equation}\label{eqt2tt}
T_n= Z({\bf t})+Z({\bf t_-})^T,~{\bf t}=(t_{i})_{i=0}^{n-1},~{\bf t}_-=(t'_{-i})_{i=0}^{n-1},~
t'_0=0.
\end{equation}
If $T_n\in  \mathcal T_{\mu,\sigma}^{n\times n}$, then
$T_n$ has $2n-1$ pairwise independent entries in $\mathcal G_{\mu,\sigma}$. Thus 
(\ref{eqttn}) implies that


$$ ||T_n||\le ||Z({\bf t})||+||Z({\bf t_-})^T||\le 
||{\bf t}||_1+||{\bf t_-}||_1= ||(t_{i})_{i=1-n}^{n-1}||_1\le \sqrt {2n-1}~||(t_{i})_{i=1-n}^{n-1}||.$$


\noindent Recall Definition \ref{defrndm} and obtain


\begin{equation}\label{eqtn}
F_{||T_n||}(y)\ge \chi_{\mu,\sigma,2n-1}(y/\sqrt {2n-1}).
\end{equation}


Next we estimate 
 the norm $||T_n^{-1}||$ 
for 
$T_{n}\in \mathcal T_{\mu,\sigma}^{n\times n}$.


\begin{lemma}\label{leinp} \cite[Lemma A.2]{SST06}.
For a nonnegative scalar $y$, a unit vector ${\bf t}\in \mathbb R^{n\times 1}$, and a vector
 ${\bf b}\in \mathcal G_{\mu,\sigma}^{n\times 1}$, 
 we have  
$F_{|{\bf t}^T{\bf b}|}(y)\le \sqrt{\frac{2}{\pi}}\frac{y}{\sigma}$.
\end{lemma}  


\begin{remark}\label{reinp}
The latter bound is independent of $\mu$ and $n$;
it holds for any $\mu$ even if 
all coordinates of the vector ${\bf b}$ are fixed except for a
single coordinate in $\mathcal G_{\mu,\sigma}$.
\end{remark}


\begin{theorem}\label{thsigunat1}  
Given a matrix 
$T_{n}=(t_{i-j})_{i,j=1}^n\in \mathcal T_{\mu,\sigma}^{n\times n}$,
assumed to be nonsingular (cf. Section \ref{srvrm}),
write 
$p_{1}={\bf e}_1^TT_n^{-1}{\bf e}_1$.
Then $F_{1/||p_{1}T_n^{-1}||}(y)\le 2n\alpha \beta$ 
 for two random variables  $\alpha$ and $\beta$
such that 
\begin{equation}\label{eqprtinv}
F_{\alpha}(y)\le \sqrt{\frac{2n}{\pi}}\frac{y}{\sigma}~{\rm and}~
F_{\beta}(y)\le \sqrt{\frac{2n}{\pi}}\frac{y}{\sigma}~{\rm for}~y\ge 0.
\end{equation}    
\end{theorem}


\begin{proof}  
Recall from part (a) of Theorem  \ref{thgs} that
$p_{1}T_n^{-1}=Z({\bf p})Z(J{\bf q})^T-Z(Z{\bf q})Z(ZJ{\bf p})^T$.
Therefore
$||p_{1}T_n^{-1}||\le ||Z({\bf p})||~||Z(J{\bf q})^T||+||Z(Z{\bf q})||~||Z(ZJ{\bf p})^T||$
for ${\bf p}=T_n^{-1}{\bf e}_1$,  ${\bf q}=T_n^{-1}{\bf e}_n$, and $p_1={\bf p}^T{\bf e}_1$.
It follows that
$||p_{1}T_n^{-1}||\le ||Z({\bf p})||~||Z(J{\bf q})||+||Z(Z{\bf q})||~||Z(ZJ{\bf p})||$
since $||A||=||A^T||$ for all matrices $A$.
Furthermore
$||p_{1}T_n^{-1}||\le||{\bf p}||_1~||J{\bf q}||_1+||Z{\bf q}||_1~||ZJ{\bf p}||_1$
due to (\ref{eqttn}).
Clearly $||J{\bf v}||_1=||{\bf v}||_1$ and $||Z{\bf v}||_1\le ||{\bf v}||_1$
for every vector ${\bf v}$, and so (cf. (\ref{eqnorm12}))
\begin{equation}\label{eqtpq}
||p_{1}T_n^{-1}||\le 2 ||{\bf p}||_1~||{\bf q}||_1\le 2n ||{\bf p}||~||{\bf q}||.
\end{equation}
  
By definition the vector ${\bf p}$ is orthogonal
to the vectors $T_n{\bf e}_2,\dots,T_n{\bf e}_n$, 
whereas ${\bf p}^TT_n{\bf e}_1=1$ (cf. \cite{SST06}).
Consequenty the vectors $T_n{\bf e}_2,\dots,T_n{\bf e}_n$
uniquely define the vector 
 ${\bf u}={\bf p}/||{\bf p}||$,
whereas 
$|{\bf u}^TT_n{\bf e}_1|=1/||{\bf p}||$.
The last coordinate $t_{n-1}$ of the vector $T_n{\bf e}_1$
is independent of the vectors $T_n{\bf e}_2,\dots,T_n{\bf e}_n$
and consequently of the vector ${\bf u}$. 
Apply 
 Remark \ref{reinp} to estimate the cdf of the random 
variable $\alpha=1/||{\bf p}||=|{\bf u}^TT_n{\bf e}_1|$  
 and obtain that
$F_{\alpha}(y)\le  \sqrt{\frac{2n}{\pi}}\frac{y}{\sigma}$ for $y\ge 0$.

Likewise the $n-1$ column vectors $T{\bf e}_1,\dots,T_{n-1}$
 define the vector ${\bf v}=\beta{\bf q}$ for 
$\beta=1/||{\bf q}||=|{\bf v}^TT_n{\bf e}_n|$.
The first coordinate $t_{1-n}$ of the vector $T_n{\bf e}_n$
is independent of the vectors $T{\bf e}_1,\dots,T_{n-1}$
and consequently of the vector ${\bf v}$. 
Apply 
 Remark \ref{reinp} to
estimate the cdf of the random 
variable $\beta$ and obtain that
$F_{\beta}(y)\le  \sqrt{\frac{2n}{\pi}}\frac{y}{\sigma}$ for $y\ge 0$.
Finally combine these bounds on the cdfs $F_{\alpha}(y)$ and 
$F_{\beta}(y)$ with (\ref{eqtpq}).
\end{proof}
 

By applying parts (b) and (c)
of Theorem  \ref{thgs} instead of its part (a),
we similarly
deduce the
bounds $||v_0T_{n+1}^{-1}||\le 2\alpha\beta$ and
 $||v_nT_{n+1}^{-1}||\le 2\alpha\beta$
for two pairs of random variables $\alpha$ and $\beta$
that
satisfy (\ref{eqprtinv}) for $n+1$ replacing $n$.
We have  $p_{1}=\frac{\det T_{n-1}}{\det T_{n}}$,
$v_0=\frac{\det T_n}{\det T_{n+1}}$, and
$v_n=\frac{\det T_{0,1}}{\det T_{n+1}}$
for $T_{0,1}=(t_{i-j})_{i=0,j=1}^{n-1,n}$.
Next we bound the 
geometric means 
of the 
ratios 
$|\frac{\det T_{h+1}}{\det T_{h}}|$
for
$h=1,\dots,k-1$. 
 $1/|p_1|$ and  $1/|v_0|$
are such  ratios for $k=n-1$ and $k=n$,
respectively,
whereas the  ratio  $1/|v_n|$ is similar to 
$1/|v_0|$, under slightly distinct notation. 



\begin{theorem}\label{thhdmr} 
 Let $T_h\neq O$ denote $h\times h$ matrices
for $h=1,\dots,k$  
whose entries have absolute values at most $t$
for a fixed scalar or random variable $t$, e.g. for $t=||T||$.
Furthermore let $T_1=(t)$.
Then the geometric mean $(\prod_{h=1}^{k-1}|\frac{\det T_{h+1}}{\det T_{h}}|)^{1/(k-1)}=\frac{1}{t}|\det T_{k}|^{1/(k-1)}$
is at most $k^{\frac{1}{2}(1+\frac{1}{k-1})}t$.
\end{theorem}


\begin{proof}
The theorem follows from 
Hadamard's upper bound
$|\det M|\le k^{k/2}t^k$, which holds
for any $k\times k$ matrix $M=(m_{i,j})_{i,j=1}^k$
with $\max_{i,j=1}^k|m_{i,j}|\le t$.
\end{proof}
 
The theorem says that
the geometric mean of the ratios $|\det T_{h+1}/\det T_{h}|$
for 
$h=1,\dots,k-1$
 is not greater than $k^{0.5+\epsilon(k)}t$
where $\epsilon(k)\rightarrow 0$ as $k\rightarrow \infty$.
Furthermore if
$T_n\in \mathcal T_{\mu,\sigma}^{n\times n}$
we can write
$t=||T||$  
 and 
 apply (\ref{eqtn}) to bound the cdf of $t$. 


\section{Extremal singular values of Gaussian random circulant matrices}\label{scgrcm}


Next we estimate the norms of a random Gaussian $f$-circulant matrix 
and its inverse. 


\begin{theorem}\label{thcircsing}
Assume $y\ge 0$ and a circulant $n\times n$ matrix $T=Z_1({\bf v})$ 
for ${\bf v}\in \mathcal G_{\mu,\sigma}^{n\times 1}$. Then
 
(a) $F_{||T||}(y)\ge \chi_{\mu,\sigma,n} (\sqrt {\frac{2}{n}}y)$
for  $\chi_{\mu,\sigma,n}(y)$ in Definition \ref{defchi} and
(b) $F_{1/||T^{-1}||}(y)\le \sqrt{\frac{2}{\pi}} \frac{ny}{\sigma}$.
\end{theorem} 


\begin{proof}
For the matrix $T=Z_1({\bf v})$
we have both equation (\ref{eqt2tt}) and the bound 
$||{\bf t_-}||_1\le||{\bf t}||_1$,
and so $||T||_1\le 2||{\bf t}||_1$. Now
part (a) of the theorem follows similarly to (\ref{eqtn}).
To prove part (b)
recall
 Theorem \ref{thcpw} and write 
$B=\Omega T\Omega^{-1}=D({\bf u})$,
${\bf u}=(u_i)_{i=0}^{n-1}=\Omega {\bf v}$. We have
$\sigma_j(T)=\sigma_j(B)$ for all $j$ because
$\frac{1}{\sqrt n}\Omega$ 
and $\sqrt n\Omega^{-1}$ are unitary matrices.
By combining the equations $u_i={\bf e}_i^T\Omega{\bf v}$, the bounds 
$||\Re ({\bf e}_i^T\Omega)||\ge 1$ for all $i$,
and Lemma \ref{leinp}, deduce that 
$F_{|\Re (u_i)|}(y)\le  \sqrt{\frac{2}{\pi}} \frac{y}{\sigma}$
for $i=1,\dots,n$.   
We have  $F_{\sigma_n(B)}(y)=F_{\min_i|u_i|}(y)$ because 
$B=\diag(u_i)_{i=0}^{n-1}$, and
clearly $|u_i|\ge |\Re (u_i)|$.
\end{proof}




\begin{remark}\label{retcond}
Our extensive experiments suggest that 
the estimates of Theorem \ref{thcircsing} are  overly pessimistic
 (cf. Table \ref{tabcondcirc}).
\end{remark}


Combining Theorem  \ref{thcpw} with minimax property (\ref{eqminmax}) implies that 
$$\frac{1}{g(f)}\sigma_j(Z_1({\bf v}))\le \sigma_j(Z_f({\bf v}))\le g(f) \sigma_j(Z_1({\bf v}))$$
for all vectors ${\bf v}$, scalars $f\neq 0$,
$g(f)=\max\{|f|^2,{1/|f|^2}\}$, and $j=1,\dots,n$. Thus we can readily extend
the estimates of Theorem \ref{thcircsing} to $f$-circulant matrices for $f\neq 0$.
In particular Gaussian random $f$-circulant matrices 
 tend to be 
well conditioned unless  $f\approx 0$ or $1/f\approx 0$.


\section{Numerical Experiments}\label{sexp}

 
Our numerical experiments with random general, Hankel, Toeplitz and circulant matrices 
have been performed in the Graduate Center of the City University of New York 
on a Dell server with a dual core 1.86 GHz
Xeon processor and 2G memory running Windows Server 2003 R2. The test
Fortran code was compiled with the GNU gfortran compiler within the Cygwin
environment.  Random numbers were generated with the random\_number
intrinsic Fortran function, assuming the uniform probability distribution 
over the range $\{x:-1 \leq x < 1\}$.  The tests have been designed by the first author 
and performed by his coauthors.


We have computed the condition numbers of random general $n\times n$ matrices for 
$n=2^k$, $k=5,6,\dots,$ with entries sampled in the range $[-1,1)$ as well as 
complex general, Toeplitz, and circulant matrices 
whose entries had real and imaginary parts sampled at random in the same range $[-1,1)$. 
We performed 100 tests for each class of inputs, each dimension $n$,  and each nullity $r$.
 Tables \ref{tab01}--\ref{tabcondcirc} display
 the test results. The last four columns of each table 
display the average (mean), minimum, maximum, and standard deviation
of the computed condition numbers of the input matrices, respectively. Namely we 
computed
the values $\kappa (A)=||A||~||A^{-1}||$ for general, Toeplitz, and circulant matrices $A$ and
the values $\kappa_1 (A)=||A||_1~||A^{-1}||_1$ for Toeplitz matrices $A$.
We computed and displayed in Table \ref{tabcondtoep} the 1-norms of 
Toeplitz matrices and their inverses rather than their 2-norms 
to facilitate the computations in the case of inputs of large sizes.
Relationships (\ref{eqnorm12}) link 
 the 1-norms and 2-norms to one another, but 
the empirical data in 
Table \ref{nonsymtoeplitz} consistently show 
even  closer links,
in all cases of
general, Toeplitz, and circulant  $n\times n$ matrices $A$ where
$n=32,64,\dots, 1024$. 


\begin{table}[h]
\caption{The norms of random general, Toeplitz and circulant $n\times n$ matrices and of their inverses}
\label{nonsymtoeplitz}
 \begin{center}
\begin{tabular}{|c|c|c|c|c|c|c|c|}
\hline
\textbf{matrix $A$}&\textbf{$n$}&\textbf{$||A||_1$}&\textbf{$||A||_2$}&\textbf{$\frac{||A||_1}{||A||_2}$}&\textbf{$||A^{-1}||_1$}&\textbf{$||A^{-1}||_2$}&\textbf{$\frac{||A^{-1}||_1}{||A^{-1}||_2}$}\\\hline
General & $32$ & $1.9\times 10^{1}$ & $1.8\times 10^{1}$ & $1.0\times 10^{0}$ & $4.0\times 10^{2}$ & $2.1\times 10^{2}$ & $1.9\times 10^{0}$ \\ \hline
General & $64$ & $3.7\times 10^{1}$ & $3.7\times 10^{1}$ & $1.0\times 10^{0}$ & $1.2\times 10^{2}$ & $6.2\times 10^{1}$ & $2.0\times 10^{0}$ \\ \hline
General & $128$ & $7.2\times 10^{1}$ & $7.4\times 10^{1}$ & $9.8\times 10^{-1}$ & $3.7\times 10^{2}$ & $1.8\times 10^{2}$ & $2.1\times 10^{0}$ \\ \hline
General & $256$ & $1.4\times 10^{2}$ & $1.5\times 10^{2}$ & $9.5\times 10^{-1}$ & $5.4\times 10^{2}$ & $2.5\times 10^{2}$ & $2.2\times 10^{0}$ \\ \hline
General & $512$ & $2.8\times 10^{2}$ & $3.0\times 10^{2}$ & $9.3\times 10^{-1}$ & $1.0\times 10^{3}$ & $4.1\times 10^{2}$ & $2.5\times 10^{0}$ \\ \hline
General & $1024$ & $5.4\times 10^{2}$ & $5.9\times 10^{2}$ & $9.2\times 10^{-1}$ & $1.1\times 10^{3}$ & $4.0\times 10^{2}$ & $2.7\times 10^{0}$ \\ \hline
Toeplitz & $32$ & $1.8\times 10^{1}$ & $1.9\times 10^{1}$ & $9.5\times 10^{-1}$ & $2.2\times 10^{1}$ & $1.3\times 10^{1}$ & $1.7\times 10^{0}$ \\ \hline
Toeplitz & $64$ & $3.4\times 10^{1}$ & $3.7\times 10^{1}$ & $9.3\times 10^{-1}$ & $4.6\times 10^{1}$ & $2.4\times 10^{1}$ & $2.0\times 10^{0}$ \\ \hline
Toeplitz & $128$ & $6.8\times 10^{1}$ & $7.4\times 10^{1}$ & $9.1\times 10^{-1}$ & $1.0\times 10^{2}$ & $4.6\times 10^{1}$ & $2.2\times 10^{0}$ \\ \hline
Toeplitz & $256$ & $1.3\times 10^{2}$ & $1.5\times 10^{2}$ & $9.0\times 10^{-1}$ & $5.7\times 10^{2}$ & $2.5\times 10^{2}$ & $2.3\times 10^{0}$ \\ \hline
Toeplitz & $512$ & $2.6\times 10^{2}$ & $3.0\times 10^{2}$ & $8.9\times 10^{-1}$ & $6.9\times 10^{2}$ & $2.6\times 10^{2}$ & $2.6\times 10^{0}$ \\ \hline
Toeplitz & $1024$ & $5.2\times 10^{2}$ & $5.9\times 10^{2}$ & $8.8\times 10^{-1}$ & $3.4\times 10^{2}$ & $1.4\times 10^{2}$ & $2.4\times 10^{0}$ \\ \hline   
Circulant & $32$ & $1.6\times 10^{1}$ & $1.8\times 10^{1}$ & $8.7\times 10^{-1}$ & $9.3\times 10^{0}$ & $1.0\times 10^{1}$ & $9.2\times 10^{-1}$ \\ \hline
Circulant & $64$ & $3.2\times 10^{1}$ & $3.7\times 10^{1}$ & $8.7\times 10^{-1}$ & $5.8\times 10^{0}$ & $6.8\times 10^{0}$ & $8.6\times 10^{-1}$ \\ \hline
Circulant & $128$ & $6.4\times 10^{1}$ & $7.4\times 10^{1}$ & $8.6\times 10^{-1}$ & $4.9\times 10^{0}$ & $5.7\times 10^{0}$ & $8.5\times 10^{-1}$ \\ \hline
Circulant & $256$ & $1.3\times 10^{2}$ & $1.5\times 10^{2}$ & $8.7\times 10^{-1}$ & $4.7\times 10^{0}$ & $5.6\times 10^{0}$ & $8.4\times 10^{-1}$ \\ \hline
Circulant & $512$ & $2.6\times 10^{2}$ & $3.0\times 10^{2}$ & $8.7\times 10^{-1}$ & $4.5\times 10^{0}$ & $5.4\times 10^{0}$ & $8.3\times 10^{-1}$ \\ \hline
Circulant & $1024$ & $5.1\times 10^{2}$ & $5.9\times 10^{2}$ & $8.7\times 10^{-1}$ & $5.5\times 10^{0}$ & $6.6\times 10^{0}$ & $8.3\times 10^{-1}$ \\ \hline
\end{tabular}
\end{center}
\end{table}


\begin{table}[h]
  \caption{The condition numbers $\kappa (A)$ of random $n\times n$ matrices $A$}
  \label{tab01}
  \begin{center}
    \begin{tabular}{| c | c | c | c | c | c |}
    \hline
\bf{$n$}&\bf{input} & \bf{min}  &\bf{max} &\bf{mean} &\bf{std} \\ \hline

$ 32   $ & $ {\rm real} $ & $2.4\times 10^{1}$ & $1.8\times 10^{3}$ & $2.4\times 10^{2}$ & $3.3\times 10^{2}$ \\ \hline
$ 64   $ & $ {\rm real} $ & $4.6\times 10^{1}$ & $1.1\times 10^{4}$ & $5.0\times 10^{2}$ & $1.1\times 10^{3}$ \\ \hline
$ 128  $ & $ {\rm real} $ & $1.0\times 10^{2}$ & $2.7\times 10^{4}$ & $1.1\times 10^{3}$ & $3.0\times 10^{3}$ \\ \hline
$ 256  $ & $ {\rm real} $ & $2.4\times 10^{2}$ & $8.4\times 10^{4}$ & $3.7\times 10^{3}$ & $9.7\times 10^{3}$ \\ \hline
$ 512  $ & $ {\rm real} $ & $3.9\times 10^{2}$ & $7.4\times 10^{5}$ & $1.8\times 10^{4}$ & $8.5\times 10^{4}$ \\ \hline
$ 1024 $ & $ {\rm real} $ & $8.8\times 10^{2}$ & $2.3\times 10^{5}$ & $8.8\times 10^{3}$ & $2.4\times 10^{4}$ \\ \hline
$ 2048 $ & $ {\rm real} $ & $2.1\times 10^{3}$ & $2.0\times 10^{5}$ & $1.8\times 10^{4}$ & $3.2\times 10^{4}$ \\ \hline
    \end{tabular}
  \end{center}
\end{table}

\begin{table}[h]
\caption{The condition numbers $\kappa_1 (A)=\frac{||A||_1}{||A^{-1}||_1}$ of random Toeplitz 
 $n\times n$ matrices $A$}
\label{tabcondtoep}
\begin{center}
\begin{tabular}{|c|c|c|c|c|}
\hline
\textbf{$n$}&\textbf{min}&\textbf{mean}&\textbf{max}&\textbf{std}\\\hline
$256$ & $9.1\times 10^{2}$ & $9.2\times 10^{3}$ & $1.3\times 10^{5}$ & $1.8\times 10^{4}$  \\ \hline
$512$ & $2.3\times 10^{3}$ & $3.0\times 10^{4}$ & $2.4\times 10^{5}$ & $4.9\times 10^{4}$  \\ \hline
$1024$ & $5.6\times 10^{3}$ & $7.0\times 10^{4}$ & $1.8\times 10^{6}$ & $2.0\times 10^{5}$ \\ \hline
$2048$ & $1.7\times 10^{4}$ & $1.8\times 10^{5}$ & $4.2\times 10^{6}$ & $5.4\times 10^{5}$ \\ \hline
$4096$ & $4.3\times 10^{4}$ & $2.7\times 10^{5}$ & $1.9\times 10^{6}$ & $3.4\times 10^{5}$ \\ \hline
$8192$ & $8.8\times 10^{4}$ & $1.2\times 10^{6}$ & $1.3\times 10^{7}$ & $2.2\times 10^{6}$ \\ \hline
\end{tabular}
\end{center}
\end{table}

\begin{table}[h]
\caption{The condition numbers $\kappa (A)$ of random circulant  $n\times n$ matrices $A$}
\label{tabcondcirc}
\begin{center}
\begin{tabular}{|c|c|c|c|c|}
\hline
\textbf{$n$}&\textbf{min}&\textbf{mean}&\textbf{max}&\textbf{std}\\\hline
$256$ & $9.6\times 10^{0}$ & $1.1\times 10^{2}$ & $3.5\times 10^{3}$ & $4.0\times 10^{2}$ \\ \hline
$512$ & $1.4\times 10^{1}$ & $8.5\times 10^{1}$ & $1.1\times 10^{3}$ & $1.3\times 10^{2}$ \\ \hline
$1024$ & $1.9\times 10^{1}$ & $1.0\times 10^{2}$ & $5.9\times 10^{2}$ & $8.6\times 10^{1}$ \\ \hline
$2048$ & $4.2\times 10^{1}$ & $1.4\times 10^{2}$ & $5.7\times 10^{2}$ & $1.0\times 10^{2}$ \\ \hline
$4096$ & $6.0\times 10^{1}$ & $2.6\times 10^{2}$ & $3.5\times 10^{3}$ & $4.2\times 10^{2}$ \\ \hline
$8192$ & $9.5\times 10^{1}$ & $3.0\times 10^{2}$ & $1.5\times 10^{3}$ & $2.5\times 10^{2}$ \\ \hline
$16384$ & $1.2\times 10^{2}$ & $4.2\times 10^{2}$ & $3.6\times 10^{3}$ & $4.5\times 10^{2}$ \\ \hline
$32768$ & $2.3\times 10^{2}$ & $7.5\times 10^{2}$ & $5.6\times 10^{3}$ & $7.1\times 10^{2}$ \\ \hline
$65536$ & $2.4\times 10^{2}$ & $1.0\times 10^{3}$ & $1.2\times 10^{4}$ & $1.3\times 10^{3}$ \\ \hline
$131072$ & $3.9\times 10^{2}$ & $1.4\times 10^{3}$ & $5.5\times 10^{3}$ & $9.0\times 10^{2}$ \\ \hline
$262144$ & $6.3\times 10^{2}$ & $3.7\times 10^{3}$ & $1.1\times 10^{5}$ & $1.1\times 10^{4}$ \\ \hline
$524288$ & $8.0\times 10^{2}$ & $3.2\times 10^{3}$ & $3.1\times 10^{4}$ & $3.7\times 10^{3}$ \\ \hline
$1048576$ & $1.2\times 10^{3}$ & $4.8\times 10^{3}$ & $3.1\times 10^{4}$ & $5.1\times 10^{3}$ \\ \hline   
\end{tabular}
\end{center}
\end{table}

\section{Implicit empirical support of the estimates of 
Sections \ref{scgrtm} and  \ref{scgrcm}}\label{simprel}

The papers \cite{PQa} and \cite{PQZa} describe successful
 applications of randomized 
circulant and Toeplitz multipliers to some fundamental matrix computations.
These applications were bound to fail if the multipliers were ill conditioned,
and so the success
gives some implicit empirical support to our probabilistic estimates of 
Sections \ref{scgrtm} and  \ref{scgrcm} and motivates the effort
for proving these estimates.

Namely it is well known that Gaussian elimination with no pivoting fails numerically
where the input matrix has an ill conditioned leading block,
even if the matrix itself is nonsingular and well conditioned.
In our extensive tests in \cite{PQa} and \cite{PQZa} 
we consistently fixed this problem by means of multiplication
by random circulant matrices. This implies 
that the random circulant matrices tend to be nonsingular and well conditioned
for otherwise the products  would be singular or ill conditioned.

Likewise in other tests in \cite{PQZa} the column sets of 
the products $A^TG$ of an $n\times m$ matrix $A^T$ having a numerical
rank $\rho$ by 
random Toeplitz $m\times \rho$ multipliers consisently  
approximated some bases for the singular spaces   
associated with the $\rho$ largest singular vaues of the matrix $A$,
and this was readily extended to computing a rank-$\rho$  approximation 
of the matrix $A$,
which is a fundamental task of matrix computations
\cite{HMT11}. Then again one can immediately observe that
 these tests would have failed numerically
if the multipliers and consequently the products were ill conditioned.

\section{Conclusions}\label{srel}


Estimating the condition numbers of random structured matrices
is a well known challenge (cf.  \cite{SST06}).
We deduce such estimates for Gaussian random Toeplitz and circulant
matrices.
The former estimates can be surprising because
the condition numbers
 grow exponentially in $n$ as $n\rightarrow \infty$
for some large and important classes
of  $n\times n$ 
Toeplitz matrices \cite{BG05}, 
whereas we prove the opposit for 
Gaussian random Toeplitz matrices.
Our formal estimates are in good accordance 
with our numerical tests,
except that circulant matrices tended to be even
better conditioned in the tests than according to our formal
study.  
The study of the condition number of Hankel matrices
is immediately reduced to the study for Toeplitz matrices
and vice versa. Can
our progress be extended to other important classes
of structured matrices?

$~$

{\bf Acknowledgements:}
Our research has been supported by NSF Grant CCF--1116736 and
PSC CUNY Awards 64512--0042 and 65792--0043.


\end{document}